\documentclass{elsarticle}

\usepackage{amsmath,amssymb,amsfonts,amscd}
\usepackage[all]{xy}
\usepackage{xcolor}

\newtheorem{theorem}{Theorem}[section]
\newtheorem{definition}[theorem]{Definition}
\newtheorem{lemma}[theorem]{Lemma}
\newtheorem{corollary}[theorem]{Corollary}

\newtheorem{example}[theorem]{Example}
\newtheorem{remark}[theorem]{Remark}
\newproof{proof}{\it Proof}

\begin{document}

\begin{frontmatter}



\title{Parts  formulas involving the Fourier--Feynman transform 
associated with Gaussian process\\   on Wiener space}

\author[label3]{Seung Jun Chang}
\ead{sejchang@dankook.ac.kr}
\author[label3]{Jae Gil Choi\corref{cor1}}
\ead{jgchoi@dankook.ac.kr}

\address[label3]{Department of Mathematics,  
                 Dankook University,
                 Cheonan 330-714, 
                 Korea}
\cortext[cor1]{Corresponding author}

 
\begin{abstract}
In this paper,
using a very general Cameron--Storvick theorem on the Wiener space $C_0[0,T]$,
we establish various integration by parts formulas involving  generalized analytic 
Feynman integrals, generalized  analytic Fourier--Feynman transforms, and the first 
variation (associated with  Gaussian processes) of functionals $F$  on $C_0[0,T]$ 
having the form $F(x)=f(\langle{\alpha_1,x}\rangle, \ldots, \langle{\alpha_n,x}\rangle)$
for scale almost every $x\in C_0[0,T]$, where $\langle{\alpha,x}\rangle$  denotes 
the Paley--Wiener--Zygmund stochastic integral $\int_0^T \alpha(t)dx(t)$, 
and $\{\alpha_1,\ldots,\alpha_n\}$ is an orthogonal set of nonzero functions 
in $L_2[0,T]$. The Gaussian processes used in this paper are  not stationary.
\end{abstract}

\begin{keyword}
Cameron--Storvick theorem \sep 
Gaussian process \sep
generalized analytic Feynman integral \sep
generalized  analytic  Fourier--Feynman transform \sep
first  variation.

\vspace{.3cm}
\MSC[2010]  28C20 \sep 46G12 \sep 42B10 \sep 60G15   

\end{keyword}

\end{frontmatter}


\setcounter{equation}{0}
\section{Introduction}\label{sec:intro}

The theory of the Fourier--Wiener transform suggested by
Cameron and Martin \cite{Cameron45,CM45,CM47-Duke} about 75 years ago
now is  playing more and more  significant role in infinite dimensional analysis,
Feynman integration theory, and applications in mathematical physics.
The Fourier--Wiener transform and several analogies which are more exquisite 
have been improved in various research fields on infinite dimensional Banach spaces.
For instance, the analytic Fourier--Feynman transform 
\cite{Brue72,CS76,CKY00-IT,HPS95,HPS96,HPS01,JS79-MMJ,KK11,KYC10,KKPS99,PSS98-RCMP,PSS98-RMJ},
the sequential  Fourier--Feynman transform \cite{CS83,CS85,CCKSY08}, and 
the integral transform  \cite{CKY00-NFAO,CCS10,KKS04,Leeyj82,Leeyj87}   are developed by many authors.
Most of the topics are concentrated on classical and abstract  Wiener spaces.

\par 
Let $C_0[0,T]$ denote  the  one-parameter     Wiener space, that is, the 
space of  all  real-valued continuous functions $x$ on $[0,T]$ with $x(0)=0$.
Let $\mathcal{M}$ denote the class of  all Wiener measurable subsets of 
$C_0[0,T]$ and let $\mathfrak{m}$ denote  the  Wiener measure. Then, as  
it  is  well-known,  $(C_0[0,T],\mathcal{M},\mathfrak{m})$ is
a complete  measure space.

\par
In the theory of infinite dimensional analysis,
the integration by parts formula is also one
of the fundamental tools to analyze the integration of functionals on the 
infinite dimensional spaces.
In \cite{Cameron51}, Cameron  derived an integration by parts formula for the Wiener 
measure $\mathfrak{m}$. This is the first infinite dimensional integration by parts 
formula.   In \cite{Donsker}, Donsker also established this formula using a different 
method, and applied it to study Fr\'echet--Volterra differential equations.  The 
integration by  parts formula on $C_0[0,T]$  introduced in \cite{Cameron51} was 
improved in \cite{CS91,PS98-PanAmer,PSS98-RCMP} to study the  parts formulas involving the analytic 
Feynman integral and the analytic Fourier--Feynman transform (henceforth FFT). Since then the 
parts formula for the analytic Feynman integral is called the Cameron--Storvick 
theorem by many mathematicians.

\par
The concept of the generalized Wiener integral (namely,  the Wiener integral 
associated with Gaussian paths) and the generalized analytic Feynman integral 
(namely,  the analytic Feynman  integral  associated with  Gaussian paths) on 
$C_0[0,T]$ were introduced by Chung, Park and Skoug \cite{CPS93}, and further 
developed in \cite{PS91,PS95}. In \cite{CPS93,PS91,PS95}, the generalized Wiener 
integral  was defined by the Wiener integral
\begin{equation}\label{eq:idea01}
\int_{C_0[0,T]}F (\mathcal{Z}_h(x,\cdot) ) \mathfrak{m}(dx),
\end{equation}
where $\mathcal{Z}_h(x,\cdot)$ is a Gaussian path given by the stochastic integral 
\begin{equation}\label{eq:Z-process-initial}
\mathcal{Z}_h(x,t) =\int_0^th(s)dx(s) \mbox{ with  } h\in L_2[0,T].
\end{equation}
For a precise definition of this stochastic integral, see Section \ref{sec:pre} below.
Also the concept of the generalized analytic Feynman integral and the   generalized  
analytic FFT (henceforth GFFT) were more developed based on the  generalized Wiener integral \eqref{eq:idea01},
see \cite{CCKSY05,CC17-JMAA,CC18-CPAA,CC18-BJMA,CC19-sub,CCC17,CSC12,HPS97}. If we choose $h\equiv 1$ 
in \eqref{eq:Z-process-initial}, as a constant function,  the generalized Wiener integral 
\eqref{eq:idea01} reduces  an ordinary Wiener 
integral, i.e., 
\[
\int_{C_0[0,T]}F (\mathcal{Z}_1(x,\cdot) ) \mathfrak{m}(dx) 
=\int_{C_0[0,T]}F (x) \mathfrak{m}(dx) .
\]

 \par
The purpose of this paper is to establish various integration by parts formulas
involving the  generalized analytic Feynman integral 
and the  GFFT  of  functionals in non-stationary Gaussian paths $\mathcal Z_h(x,\cdot)$
given by \eqref{eq:Z-process-initial}.
In Section \ref{sec:remark} below we illustrate the importance of this topic and the
motivation of this paper.  

\par
The Wiener process used in
\cite{Brue72,Cameron45,Cameron51,CM45,CM47-Duke,CS76,CS83,CS85,CS91,CCKSY08,
CKY00-IT,CKY00-NFAO,CCS10,HPS95,HPS96,HPS01,JS79-MMJ,KK11,KKS04,KYC10,KKPS99,Leeyj82,Leeyj87, 
PS98-PanAmer,PSS98-RCMP,PSS98-RMJ} is a stationary 
process. However, the stochastic process $\mathcal Z_h$ on $C_0[0,T]$ used in this paper, 
as well as  in \cite{CCKSY05,CC17-JMAA,CC18-CPAA,CC18-BJMA,CC19-sub,CCC17,CSC12,CPS93,HPS97,PS91,PS95}, 
is non-stationary in time. 
Thus the results in this paper are quite a lot more complicated because the Gaussian processes used 
in this paper are  non-stationary  processes.  However, by choosing $h(t)\equiv 1$ on $[0,T]$
in \eqref{eq:Z-process-initial}, the process $\mathcal Z_1$  reduces to an ordinary 
Wiener process on $C_0[0,T]$, and so the expected results on $C_0[0,T]$ are immediate 
corollaries of the results in this
paper.


\setcounter{equation}{0}
\section{Preliminaries}\label{sec:pre}

\par
In this section we first present a brief background and some
well-known results about the Wiener  space $C_0[0,T]$.

\par
A subset $B$ of $C_0[0,T]$ is said to be scale-invariant measurable \cite{JS79-pacific}
provided $\rho B\in \mathcal{M}$ for all $\rho>0$, and a scale-invariant measurable
set $N$ is said to be scale-invariant null provided  $\mathfrak{m}(\rho N)=0$ for
all $\rho>0$. A property that holds except on  a scale-invariant null set is said
to  hold  scale-invariant almost everywhere (s-a.e.). A functional $F$ is said
to be scale-invariant measurable provided $F$ is defined on a scale-invariant
measurable set and $F(\rho\,\cdot\,)$ is Wiener-measurable for every $\rho> 0$.
If two functionals $F$ and $G$ are equal s-a.e.,  we write $F\approx G$. 

\par
The Paley--Wiener--Zygmund (henceforth PWZ) stochastic integral \cite{PWZ33} plays a key
role throughout this paper. Let $\{\phi_n\}_{n=1}^{\infty}$ be a complete orthonormal
set in $L_2[0,T]$, each of whose elements is  of bounded variation on $[0,T]$. Then
for each $v\in L_2[0,T]$, the PWZ stochastic integral $\langle{v,x}\rangle$  is
defined by the formula
\[
\langle{v,x}\rangle
=\lim\limits_{n\to \infty}
\int_0^T\sum\limits_{j=1}^n
(v,\phi_j)_2 \phi_j(t)d x(t)
\]
for all $x\in C_0[0,T]$ for which the limit exists, where $(\cdot,\cdot)_2$
denotes the $L_2$-inner product. For each  $v\in L_2[0,T]$, the limit defining
the PWZ  stochastic  integral  $\langle{v,x}\rangle$ is essentially independent
of the choice of the complete  orthonormal set  $\{\phi_n\}_{n=1}^{\infty}$
and it exists for s-a.e. $x\in C_0[0,T]$.
If $v$ is of bounded variation on $[0,T]$  then $\langle{v,x}\rangle$ equals the
Riemann--Stieltjes integral $\int_0^T v(t)dx(t)$  for s-a.e. $x\in C_0 [0,T]$,
and   for  each  $v$  in $L_2[0,T]$,  
$\langle{v,\cdot}\rangle$  is a  Gaussian random variable on $C_0[0,T]$ 
with mean zero and variance $\|v\|_2^2$. If $\{\alpha_1,\ldots,\alpha_n\}$ is 
an orthogonal set of functions in $L_2[0,T]$, then the random variables, 
$\langle{\alpha_j,x}\rangle$'s, are independent. For a more  detailed study of the PWZ 
stochastic integral, see \cite{JS81-JFA,PS88}.

\par
Throughout this paper we let
\[
\begin{aligned}
\mathrm{Supp}_2[0,T]
&=\{h\in L_2[0,T]:  m_L (\mathrm{supp}(h))=T\}\\
&=\{h\in L_2[0,T]:  h\ne 0\,\,\,\, m_L\mbox{-a.e. on } [0,T]\}
\end{aligned}
\]
and
\[
\mathrm{Supp}_{\infty}[0,T]
=\{h\in L_{\infty}[0,T]:  h\ne 0\,\,\,\, m_L\mbox{-a.e. on } [0,T]\}
\]
where  $m_L$ denotes  Lebesgue measure on $[0,T]$.
We note that $\mathrm{Supp}_{\infty}[0,T] \subset\mathrm{Supp}_{2}[0,T]$,
and for any $h\in \mathrm{Supp}_{2}[0,T]$, $\|h\|_2>0$.

\par
Given any function  $h$ in $\mathrm{Supp}_2[0,T]$,
let $\mathcal{Z}_h: C_0[0,T]\times[0,T]\to\mathbb R$
be the stochastic process given by
\begin{equation}\label{eq:g-process}
\mathcal{Z}_h(x,t)
=\langle{h\chi_{[0,t]},x}\rangle,
\end{equation}
where $\chi_{[0,t]}$ denotes the indicator function of
the set $[0,t]$. Next, let $\beta_h(t) =\int_0^t h^2(u)du$.
The stochastic process  $\mathcal{Z}_h$ on $C_0[0,T]\times[0,T]$
is a Gaussian process with mean zero and covariance function
\[
\int_{C_0[0,T]} \mathcal{Z}_h(x,s)\mathcal{Z}_h(x,t)\mathfrak{m}(dx)
=\beta_h({\min \{s,t\}}) .
\]
In addition, by \cite[Theorem 21.1]{Yeh73}, $\mathcal{Z}_h (\cdot, t)$
is stochastically  continuous in $t$ on $[0,T]$.  Also, for any
$h_1,h_2 \in\mathrm{Supp}_{2}[0,T]$,
\[
\int_{C_0[0,T]}
\mathcal{Z}_{h_1}(x,s)\mathcal{Z}_{h_2}(x,t)
\mathfrak{m}(dx)
=\int_{0}^{\min\{s,t\}}h_1(u)h_2(u) d u.
\]
Of course, as discussed in Section \ref{sec:intro} above, if $h(t)\equiv 1$ on $[0,T]$,
then the process $\mathcal W$ on $C_0[0,T]\times[0,T]$ given by
$(w,t)\stackrel{\mathcal W}{\longrightarrow}
\mathcal W_t(x)=\mathcal  Z_1 (x,t)=x(t)$
is a Wiener process (standard Brownian motion).
We note that the coordinate  process $\mathcal Z_1$ is stationary
in time, whereas the stochastic process  $\mathcal{Z}_h$ generally
is not. For more detailed studies on the stochastic process  $\mathcal{Z}_h$,
see \cite{CCKSY05,CC17-JMAA,CC18-CPAA,CC18-BJMA,CC19-sub,CCC17,CSC12,CPS93,HPS97,PS91,PS95}.

\par
 If $h\in\mathrm{Supp}_2[0,T] \cap BV[0,T]$, then for all $x\in C_0[0,T]$,
$\mathcal{Z}_h(x,t)$ is continuous in $t$. 
From the definition of the PWZ stochastic integral, it follows  that for each  
$v\in L_2[0,T]$  and  each  $h\in \mathrm{Supp}_{\infty}[0,T]$, 
\begin{equation}\label{eq:basic-rel}
\langle{v,\mathcal{Z}_{h}(x,\cdot)}\rangle =\langle{vh,x}\rangle
\end{equation}
for s-a.e. $x\in C_{0}[0,T]$. Thus, throughout this paper, we require $h$ 
to be in $\mathrm{Supp}_{\infty}[0,T]$ rather than simply in $\mathrm{Supp}_{2}[0,T]$. 
 
\par
Let $\mathbb C$, $\mathbb C_+$ and $\mathbb{\widetilde C}_+$ denote
the set of  complex numbers, complex numbers with positive real part
and nonzero complex numbers with nonnegative real part, respectively.
For each $\lambda \in \mathbb C$, $\lambda^{1/2}$ denotes the principal
square root of $\lambda$; i.e., $\lambda^{1/2}$ is always chosen to have
positive real part, so that
$\lambda^{-1/2}=(\lambda^{-1})^{1/2}$ is  in $\mathbb C_+$ for
all $\lambda\in\widetilde{\mathbb C}_+$.

\begin{definition}\label{def:Faynman}
Let $h$ be a function in $\mathrm{Supp}_2[0,T]$ and let $F$ be
a $\mathbb C$-valued scale-invariant measurable functional
on $C_0[0,T]$  such that
\[
\int_{C_0[0,T]}
F(\lambda^{-1/2}\mathcal Z_h(x,\cdot))\mathfrak{m}(dx)
=J(h;\lambda)
\]
exists as a finite number for all $\lambda>0$. If there exists
a function $J^* (h;\lambda)$ analytic on $\mathbb C_+$ such that
$J^*(h;\lambda)=J(h;\lambda)$ for all $\lambda>0$, then  $J^*(h;\lambda)$
is defined to be the generalized analytic  Wiener integral (associated
with the Gaussian paths   $\mathcal{Z}_h(x,\cdot)$) of $F$  over $C_0[0,T]$ with
parameter $\lambda$, and for $\lambda \in \mathbb C_+$ we write
\[
\int_{C_0[0,T]}^{\text{\rm anw}_{\lambda}}
F(\mathcal{Z}_h(x,\cdot))\mathfrak{m}(dx)
= J^*(h;\lambda).
\]
Let $q\ne 0$ be a real number and let $F$ be a functional such that
the generalized analytic  Wiener integral,
$\int_{C_0[0,T]}^{\text{\rm anw}_{\lambda}}
F(\mathcal{Z}_h(x,\cdot))\mathfrak{m}(dx)$,
exists for all $\lambda \in \mathbb C_+$. If the following limit
exists, we call it the generalized analytic Feynman integral (associated
with the Gaussian paths   $\mathcal{Z}_h(x,\cdot)$) of $F$ with parameter $q$ and
we  write
\begin{equation}\label{eq:zhFint}
\int_{C_0[0,T]}^{\mathrm{anf}_{q}}
F(\mathcal{Z}_h(x,\cdot))d\mathfrak m(x)
= \lim_{\substack{
\lambda\to -iq \\ \,\, \lambda\in \mathbb C_+}}
\int_{C_0[0,T]}^{\mathrm{anw}_{\lambda}}
F(\mathcal{Z}_h(x,\cdot))\mathfrak{m}(dx).
\end{equation}
\end{definition}

\par 
Next (see \cite{CC17-JMAA,CC18-BJMA,CC19-sub,CCC17,CSC12,HPS97}) we  state the definition of
the $L_p$ analytic $\mathcal Z_h$-GFFT
(namely, the analytic GFFT associated with the Gaussian paths 
$\mathcal Z_h(x,\cdot)$).

\begin{definition} \label{def:tpq}
Let $h$ be a nonzero function in $\mathrm{Supp}_2[0,T]$.  
For $\lambda\in\mathbb{C}_+$ and $y \in C_{0}[0,T]$, let
\[
T_{\lambda,h}(F)(y)
=\int_{C_0[0,T]}^{\text{\rm anw}_{\lambda}} 
F(y+\mathcal{Z}_h(x,\cdot)) \mathfrak{m}(dx). 
\]
For $p\in (1,2]$ we define the $L_p$ analytic $\mathcal Z_h$-GFFT, 
$T^{(p)}_{q,h}(F)$ of $F$, by the formula,
\[
T^{(p)}_{q,h}(F)(y)
=\operatorname*{l.i.m.}_{\substack{
\lambda\to -iq \\ \,\, \lambda\in \mathbb C_+}}
T_{\lambda,h} (F)(y)    				 
\]
if it exists; i.e.,  for each $\rho>0$,
\[
\lim_{\substack{
\lambda\to -iq \\ \,\, \lambda\in \mathbb C_+}}
\int_{C_{0}[0,T]}\big| T_{\lambda,h} (F)(\rho y)
   -T^{(p)}_{q, h }(F)(\rho y) \big|^{p'} 
 \mathfrak m (dy)=0
\]
where $1/p+1/p' =1$. We define the $L_1$ analytic $\mathcal Z_h$-GFFT, 
$T_{q, h }^{(1)}(F)$ of  $F$  by  the formula  
\[
T_{q, h }^{(1)}(F)(y)
= \lim_{\substack{
\lambda\to -iq \\ \,\, \lambda\in \mathbb C_+}}
T_{\lambda,h} (F)(y)
\]
for s-a.e. $y\in C_0[0,T]$ whenever this limit exists.
\end{definition}

\par
We note that for $p \in [1,2]$, $T_{q,h}^{(p)}(F)$ is  defined only s-a.e..
We also note that if $T_{q,h}^{(p)}(F)$  exists  and if $F\approx G$, then 
$T_{q,h}^{(p)}(G)$ exists  and  $T_{q,h}^{(p)}(G)\approx T_{q,h }^{(p)}(F)$.
One can see that for each $h\in \mathrm{Supp}_2[0,T]$, 
$T_{q,h}^{(p)}(F)\approx T_{q,-h}^{(p)}(F)$,
since 
\begin{equation*}\label{symmetric}
\int_{C_0[0,T]}F(x) \mathfrak{m}(dx)=\int_{C_0[0,T]}F(-x) \mathfrak{m}(dx).
\end{equation*}

\begin{remark}\label{remark:ordinary-fft}
Note that if $h(t)\equiv 1$ on $[0,T]$, 
then $\mathcal{Z}_h (x,t)=x(t)$ for all $x\in C_0[0,T]$.
In this case the generalized analytic Feynman integral given by 
equation \eqref{eq:zhFint} above and the  $L_p$ analytic  
$\mathcal Z_1$-GFFT, $T_{q,1}^{(p)}(F)$, agree with the previous 
definitions of the analytic Feynman integral and the analytic 
FFT, $T_{q}^{(p)}(F)$, see  \cite{Brue72,CS76,CS91,CKY00-IT,HPS95,HPS96,HPS01,JS79-MMJ,JS81-Pacific,%
JS81-JFA,KK09,KYC10,KKPS99,PSS98-RCMP,PSS98-RMJ}.
\end{remark}

\par
Next we give the definition of the first variation $\delta F$ of
a functional $F$. The following definition is due to by Chang, Cho,
Kim, Song and Yoo \cite{CCKSY05}.

\begin{definition} \label{def:1st-var}
Let $h_1$   and $h_2$ be  nonzero functions in $\mathrm{Supp}_2[0,T]$, let $F$ be 
a Wiener measurable functional on $C_{0}[0,T]$,
and let $w \in C_{0}[0,T]$. Then
\begin{equation}\label{eq:1st}
\begin{aligned}
\delta_{h_1,h_2} F(x|w)
&\equiv \delta  F(\mathcal Z_{h_1}(x,\cdot)|\mathcal Z_{h_2}(w,\cdot))\\
&=\frac{\partial}{\partial \mu}
F(\mathcal Z_{h_1}(x,\cdot)+\mu \mathcal Z_{h_2}(w,\cdot)) \bigg|_{\mu=0}
\end{aligned}
\end{equation}
(if it exists) is called the first variation of $F$ in the direction $w$.
\end{definition}

\begin{remark}
Setting $h_1 = h_2 \equiv  1$ on $[0,T]$, our definition of the first variation  reduces to 
the first variation studied in  \cite{Cameron51,CS91,CKY00-IT,KK09,KKS04,KKPS99,PSS98-RCMP,PSS98-RMJ}. 
That is,
\[
\delta_{1,1}F(x|w) = \delta F(x|w).
\]
\end{remark}

\par
Throughout this paper we shall always choose $w$ to be an
element of  $C_0'[0,T]$ where
\begin{equation}\label{eq:CMspace}
 \begin{aligned}
&C_0'[0,T] \\
&=\{w\in C_0[0,T]: w \mbox{ is absolutely continuous on  } [0,T]
\mbox{ with } w' \in L_2[0,T] \}.
 \end{aligned}
\end{equation}

\setcounter{equation}{0}
\section{Remark on the topic of this paper}\label{sec:remark}

 \subsection{A short survey of Cameron--Storvick theorem}\label{subsec:A}

\par
In \cite{Cameron51},  Cameron introduced the first variation (a kind of
G\^ateaux  derivative) of functionals  on  $C_0[0,T]$ and obtained a
formula involving  the Wiener integral of the first variation. In \cite{CS91},
Cameron and Storvick   established  a similar result for  the analytic
Feynman integral of  functionals on $C_0[0,T]$. They also applied their
celebrated result to establish the existence of the analytic Feynman
integral of unbounded  functionals on $C_0[0,T]$.  
We start  this section by stating the original   Cameron--Storvick theorem 
and the parts formula for the analytic Feynman integral  of functionals $F$ on $C_0[0,T]$.
To do this, in this section we consider the  ordinary `analytic Feynman integral'
(namely, the Feynman integral  associated with the Gaussian paths  $\mathcal{Z}_1(x,\cdot)$),
$\int_{C_0[0,T]}^{\mathrm{anf}_q}F(x)   \frak{m}(dx)$,  and
the Cameron--Storvick's  first variation, $\delta F \equiv \delta_{1,1}F$, 
for functionals $F$ on $C_0[0,T]$.

\begin{theorem} Let $z\in L_2[0,T]$ and let $w(t)=\int_0^t z(s)ds$.
For each $\rho>0$, let $F(\rho x)$ be Wiener integrable on $C_0[0,T]$ and let 
$F(\rho x)$ have a first variation $\delta  F(\rho x| \rho w)$ for all $x \in C_0[0,T]$
such that for some positive function $\eta (\rho)$,
\[
\sup_{|h|\le \eta(\rho)} |\delta  F(\rho x +\rho h w| \rho w)| 
\]
is Wiener integrable. Then if either member of the following equation
exists, both analytic Feynman integrals below exist, and for each
$q \in \mathbb R \setminus \{0\}$,
\begin{equation}\label{eq:CS-origin}
\int_{C_0[0,T]}^{\mathrm{anf}_q}\delta F(x|w) \mathfrak{m}(dx)
=-iq\int_{C_0[0,T]}^{\mathrm{anf}_q}  \langle{z,x}\rangle F(x)   \mathfrak{m}(dx).
\end{equation}
\end{theorem}

\begin{remark}
In \cite{CS91}, Cameron and Storvick require $z$ to be ``essentially of bounded 
variation'', but as was pointed by Cameron \cite[p.915]{Cameron51} this requirement 
can be replaced by the requirement that $z$ be ``of class $L_2[0,T]$'' since all of 
our Stieltjes integrals are interpreted  
as Paley--Wiener--Zygmund integrals.
\end{remark}

The following integration by parts formula and further applications are investigated 
in many previous researches. For instance, see \cite{KK09,PS98-PanAmer,PSS98-RCMP}.


\begin{theorem}
Let $w$ be a function in $C_0[0,T]$, and let $F$ and $G$ be  scale-invariant 
measurable functionals on $C_0[0,T]$.
Assume that the   first variations in the following equations all exist.
Then it follows that 
\begin{equation}\label{eq:siple01}
\begin{aligned}
&\int_{C_0[0,T]}^{\mathrm{anf}_q}
\big[F(x)\delta G (x| w) +\delta F (x|w) G (x)\big]\mathfrak{m}(dx)\\
&\stackrel{*}{=}-iq \int_{C_0[0,T]}^{\mathrm{anf}_q} \langle{z,x}\rangle
F(x)G (x) \mathfrak{m}(dx).
\end{aligned}
\end{equation}
where by $\stackrel{*}{=}$ we mean  that if
either side exists, both side exist and equality holds.
\end{theorem}

Using  a heuristic use of the Cameron--Storvick theorem (namely equation \eqref{eq:CS-origin}),
equation   \eqref{eq:siple01} 
is a simple consequence, because 
\[
\delta (FG)(x|w)=F(x)\delta   G (x|w)+\delta   F (x|w) G(x)
\] 
for almost all functionals $F$ and $G$ on $C_0[0,T]$.
Thus, to establish the equality in \eqref{eq:siple01}, 
the authors of the papers \cite{KK09,PS98-PanAmer,PSS98-RCMP} 
guaranteed the existences of the Feynman integrals and the first variations  
in the corresponding formulas to the equation   \eqref{eq:siple01}.
But the singularities of the  Wiener measure \cite{Cameron54,CM47-bams,JS79-pacific}, and  
the unusual behaviors of the  analytic  Feynman integral and the  analytic FFT  \cite{CS76,JS79-MMJ,JS81-Pacific}  
of functionals on $C_0[0,T]$ make  establishing various integration by parts formulas
involving  the Feynman integral and the FFT very difficult.
These  are due to the fact that the Wiener measure $\mathfrak{m}$
is not a  quasi-invariant probability measure.
It is well known that there is no quasi-invariant measure on infinite 
dimensional linear spaces (see  \cite{Yamasaki}).
Thus the translation theorem (Cameron--Martin theorem) and the Girsanov theorem on 
infinite dimensional Banach  spaces have been studied in the literature.
An essential structure hidden in the proof of the  Cameron--Storvick theorem
is based on the  Cameron--Martin translation theorem on Wiener space $C_0[0,T]$.

\subsection{Why do we use the Gaussian processes  defining the GFFT ?}

We consider the class of the (ordinary) $L_p$ analytic FFTs, $\{T_{q}^{(p)}\}_{q\in \mathbb R}$,
where the FFT $T_{0}^{(p)}$ with parameter $q=0$ denotes the identity transform,
i.e., $T_{0}^{(p)}(F) =F$ for functionals $F$ on $C_0[0,T]$.  
Then the class $\{T_{q}^{(p)}\}_{q\in \mathbb R}$ of the $L_p$ analytic FFTs
forms a commutative group  acting on various large classes of functionals on $C_0[0,T]$. 
We refer to the article \cite{HPS01} for a more detailed study of this topic.
In fact, in  \cite{HPS01}, Huffman, Park and Skoug presented the results with the class of 
the $L_1$ analytic FFTs, $\{T_{q}^{(1)}\}_{q\in \mathbb R}$, to furnish  simple illustrations
of the algebraic structure of the classes of  FFTs.
But, as commended in \cite{HPS01},  most of the results hold for the class of 
the $L_p$ FFTs with $p\in[1,2]$.

On the other hand, in \cite{CC17-JMAA,CSC12}, Chang, Choi and Skoug discovered 
new algebraic structures of the classes of the GFFT associated Gaussian processes.
Furthermore, in  \cite{CC19-sub,CCC17}, the authors investigated various relationships between 
 the GFFT  and the corresponding convolution products. 
There are many improvements and applications of subjects involving the concepts of the GFFT.
As a natural consequence work, it would be interesting to determine if the relationship between the ordinary 
FFT and the first variation could be extended to the case of the relationship between the GFFT 
 and the general first variation defined by \eqref{eq:1st}.
Thus, in this paper  we also  study other properties of the GFFT together
with the generalized first variation.

As discussed  above, the essential structure of parts formulas on Wiener space  
is based on the Cameron--Storvick theorem. Thus, to establish our parts formulas 
involving the generalized Feynman integral and the GFFT, we will present   a more  
general   Cameron--Storvick theorem   using the above notation.

\begin{theorem}[\cite{CC18-CPAA}]\label{byparts-step1}
Let  $h_1$ and $h_2$ be  functions in   $\text{\rm Supp}_2[0,T]$   and given
$z\in BV[0,T]$, let $w_{z h_1} \in C_0'[0,T]$ be defined by
\begin{equation}\label{eq:function-w-phih} 
w_{z h_1}(t)=\int_{0}^tz(s) h_1(s)ds.
\end{equation}
Let $F$ be a functional on $C_{0}[0,T]$ such that
$F(\mathcal Z_{h_1}(x,\cdot))$ is Wiener integrable over $C_{0}[0,T]$.
 Furthermore assume that for each $\rho>0$,
\begin{equation}\label{step1-condition}
\int_{C_{0}[0,T]}
\big|\delta  F(\rho \mathcal Z_{h_1}(x,\cdot)
|\rho\mathcal Z_{h_2}(w_{z h_1},\cdot))\big|\mathfrak{m}(dx)<+\infty.
\end{equation}
Then, 
 if either member of  the
following equation exists, both  generalized analytic Feynman integrals
below exist, and for each $q\in \mathbb R\setminus\{0\}$,
\begin{equation*}\label{eq:byparts-step3-show-01}
\begin{aligned}
&\int_{C_{0}[0,T]}^{\text{\rm anf}_q}
\delta F( \mathcal Z_{h_1} (x,\cdot)| \mathcal Z_{h_2}(w_{z h_1},\cdot))
\mathfrak{m}(dx)\\
&=-iq\int_{C_{0}[0,T]}^{\text{\rm anf}_q}
\langle{z,\mathcal Z_{h_2}(x,\cdot)}\rangle
F( \mathcal Z_{h_1}(x,\cdot))  \mathfrak{m}(dx).
\end{aligned}
\end{equation*}
\end{theorem}


\begin{remark}
In \cite{CC18-CPAA}, Chang and Choi require $z$ to be ``of bounded variation'', but 
this requirement also can be replaced by the requirement that $z$ be ``of class $L_2[0,T]$'' since all 
of our Stieltjes integrals are interpreted  as Paley--Wiener--Zygmund integral.
Also, the condition \eqref{step1-condition} above can be replaced with the condition:
for some $\eta>0$,
\[
\sup_{|h|\le \eta(\rho)} |\delta  F(\rho \mathcal Z_{h_1} (x,\cdot) 
+\rho h \mathcal Z_{h_2}(w_{z h_1},\cdot)| \rho \mathcal Z_{h_2}(w_{z h_1},\cdot))| 
\]
is Wiener integrable  as a function of $x$.
Thus, setting $h_1\equiv 1$ and  $h_2\equiv 1$ yields the formula \eqref{eq:CS-origin}.
\end{remark}

\setcounter{equation}{0}
\section{Cylinder functionals}

\par
Functionals that involve PWZ stochastic integrals are quite common. 
A functional $F$ on $C_{0}[0,T]$ is called a cylinder functional if there 
exists a linearly independent  set $\mathcal{V}=\{v_1,\ldots, v_m\}$ 
of nonzero functions in $L_2[0,T]$ such that
\begin{equation}\label{eq:cylinder1}
F(x)
=\psi(\langle{v_1,x}\rangle,\ldots,\langle{v_m,x}\rangle),
\quad x \in C_{0}[0,T], 
\end{equation}
where $\psi$ is a complex-valued Lebesgue measurable function on $\mathbb R^m$.

\par
It is easy to show  that for the functional $F$ of the 
form \eqref{eq:cylinder1}, there exists an orthogonal  set  
$\mathcal{A}=\{\alpha_1,\ldots,\alpha_n\}$  of nonzero functions in $L_2[0,T]$ such that 
$F$ is expressed as 
\begin{equation}\label{eq:cylinder-x}
F(x)
=f(\langle{\alpha_1,x}\rangle,\ldots,\langle{\alpha_n,x}\rangle),
\quad x \in C_{0}[0,T], 
\end{equation}
where $f$ is a complex-valued Lebesgue measurable function on $\mathbb R^n$.  
Thus, there is no loss of generality in assuming that every cylinder  functional on 
$C_{0}[0,T]$ is of the form  \eqref{eq:cylinder-x}.

\par 
For $h\in \mathrm{Supp}_{\infty}[0,T]$, let  $\mathcal{Z}_h$ be 
the Gaussian process given by \eqref{eq:g-process} above and let $F$ be 
given by equation \eqref{eq:cylinder-x}. Then  by equation \eqref{eq:basic-rel},
\[
\begin{aligned}
F(\mathcal{Z}_h (x,\cdot))
&=f(\langle{\alpha_1,\mathcal{Z}_h (x,\cdot)}\rangle,
    \ldots ,\langle{\alpha_n,\mathcal{Z}_h (x,\cdot)}\rangle)\\
&=f(\langle{\alpha_1h,x}\rangle, \ldots ,\langle{\alpha_n h,x}\rangle).
\end{aligned}
\]


\begin{remark}
Even though the  set   $\mathcal{A}=\{\alpha_1,\ldots, \alpha_n\}$ 
of nonzero functions in $L_2[0,T]$ is   orthogonal,  the subset 
$\mathcal{A}h\equiv\{\alpha h: \alpha\in \mathcal{A}\}$
of $L_2[0,T]$ need not  be  orthogonal. 
Given an orthogonal  set $\mathcal{A}=\{\alpha_1,\ldots,\alpha_n\}$   of nonzero 
functions in  $L_2[0,T]$, let $\mathcal{O}_{\mathrm{Supp}_{\infty}}(\mathcal{A})$ 
be the class of all functions  $h\in \mathrm{Supp}_{\infty}[0,T]$ such that $\mathcal{A}h$ 
is orthogonal in $L_2[0,T]$. Since $\dim L_2[0,T]=\infty$,  infinitely 
many functions  $h$  exist in $\mathcal{O}_{\mathrm{Supp}_{\infty}}(\mathcal{A})$.  
\end{remark}


\begin{example}
For any $h\in\mathbb R\setminus\{0\}$, 
$h\in\mathcal{O}_{\mathrm{Supp}_{\infty}}(\mathcal{A})$, 
as a constant function on $[0,T]$.
\end{example}
 
\begin{example}
For any orthogonal set $\mathcal{A}=\{\alpha_1,\ldots, \alpha_n\}$ of nonzero
functions in $L_2[0,T]$,
let $L(S)$ be the subspace of $L_2[0,T]$ which is spanned by 
\[
S=\left\{\alpha_i\alpha_j : 1\leq i<j\leq n \right\},
\] 
and let $L(S)^{\perp}$ be the orthogonal 
complement of $L(S)$. Let 
\[
\mathcal{P}_{\mathrm{Supp}_{\infty}}(\mathcal{A})
=\{ h\in \mathrm{Supp}_{\infty}[0,T] : h^2 \in L(S)^{\perp}\}.
\]
Since $\dim L(S)$ is finite, and $\mathrm{Supp}_{\infty}[0,T]$ is dense in $L_2[0,T]$
($\mathrm{Supp}_{\infty}[0,T]$ contains all polynomials on $[0,T]$), 
$\dim(L(S)^{\perp}\cap \mathrm{Supp}_{\infty}[0,T])=\infty$  
and so $\mathcal{P}_{\mathrm{Supp}_{\infty}}(\mathcal{A})$ has infinitely many elements.

\par 
Let $h$ be an element of $\mathcal{P}_{\mathrm{Supp}_{\infty}}(\mathcal{A})$.
It is easy to show that  $\|\alpha_j h\|_{2}>0$ for all  $j\in\{1,\ldots,n\}$.
From the definition of the $\mathcal{P}_{\mathrm{Supp}_{\infty}}(\mathcal{A})$, we see that  
for $i,j\in\{1,\ldots,n\}$ with $i\ne j$,
\[
(\alpha_i h, \alpha_j h)_{2}
=\int_0^T \alpha_i(t) \alpha_j(t) h^2(t) dt=0.
\]
From these, we see that  $\mathcal{A} h$ is an orthogonal set of functions in $L_2[0,T]$
for any $h$   in $\mathcal{P}_{\mathrm{Supp}_{\infty}}(\mathcal{A})$, i.e.,
$\mathcal{P}_{\mathrm{Supp}_{\infty}}(\mathcal{A}) \subset \mathcal{O}_{\mathrm{Supp}_{\infty}}(\mathcal{A})$. 
\end{example}

The following lemma is very useful in order to establish our parts formulas in this paper.

\begin{lemma}[Wiener Integration Theorem]\label{lemma:ch-formula}
Let $\mathcal G=\{g_1, \ldots, g_{n}\}$ be an orthogonal set of nonzero 
functions in $\mathrm{Supp}_2[0,T]$.
Let $f: \mathbb R^{n} \to \mathbb C$ be a  Lebesgue measurable function.  Then
for any $\rho>0$,
\begin{equation}\label{eq:ch-formula}
\begin{aligned}
&\int_{C_0[0,T]}f (\rho \langle{g_1,x}\rangle,\ldots,\rho \langle{g_n,x}\rangle) \mathfrak{m}(dx)  \\
&\stackrel{*}{=}
\bigg(\prod_{j=1}^{n}2\pi\rho^2 \|g_j \|_2^2 \bigg)^{-1/2}\int_{\mathbb R^{n}}
f(u_1,\ldots, u_n)\exp
\bigg\{-\sum_{j=1}^{n}\frac{u_j^2}{2\rho^2 \|g_j\|_2^2}\bigg\}du_1\cdots d u_n,
\end{aligned}
\end{equation}
where  by $\stackrel{*}{=}$ we mean  that if either side  exists, 
both sides exist and equality holds.
\end{lemma}

\par 
Let $n$ be a positive integer (fixed throughout this paper) and 
let $\mathcal A=\{\alpha_1, \ldots, \alpha_n\}$ be an orthogonal set of nonzero functions from 
$(L_2[0,T], \|\cdot\|_2)$. Let $m$ be a nonnegative integer. 
Then for $1 \le p <  +\infty$, let $\mathcal B_{\mathcal A}(p;m)$ be the space of all functionals of the
form  \eqref{eq:cylinder-x}
for s-a.e. $x \in C_0[0,T]$ where all of the $k$th-order partial derivatives
\[
f_{j_1,\ldots,j_k}(u_1,\ldots, u_n)  =  f_{j_1,\ldots,j_k} (\vec u)
\]
 of $f : \mathbb R^n \to \mathbb   R$
are continuous and in $L_p(\mathbb R^n)$ for $k\in \{0,1,\ldots,m\}$ and each $j_i \in \{1,\ldots, n\}$.
Also, let $\mathcal B_{\mathcal A}(\infty;m)$ be the space of all functionals of the form \eqref{eq:cylinder-x} 
for s-a.e. $x \in C_0[0,T]$  where for $k \in \{ 0, 1,\ldots,m\}$,
all of the $k$th-order partial derivatives $ f_{j_1,\ldots,j_k} (\vec u)$ of $f$
are in $C_0(\mathbb R^n)$, the space of bounded continuous functions on $\mathbb R^n$  
that vanish at infinity.

\setcounter{equation}{0}
\section{Integration by parts formulas for the  generalized analytic Feynman integral}\label{sec:parts-f}

In this section we establish integration by parts formulas for
the generalized analytic Feynman integral. 
We start this section with the existence of the generalized analytic Feynman integral associated with
Gaussian paths $\mathcal Z_h$ of functionals $F$ in $\mathcal B_{\mathcal A}(p;m)$.


\begin{theorem}\label{them:1-st}
Let $\mathcal A=\{\alpha_1,\ldots,\alpha_n\}$ be an orthogonal set of nonzero functions in $L_2[0,T]$,
let $p\in[1,+\infty]$ be given, let $m$ be a non-negative integer, and 
let $F\in \mathcal B_{\mathcal A}(p;m)$ be given by equation \eqref{eq:cylinder-x}.
Then for any $h\in \mathcal{O}_{\mathrm{Supp}_{\infty}}(\mathcal{A})$ and all $q\in \mathbb R\setminus\{0\}$,
the generalized analytic Feynman integral associated
with the Gaussian paths   $\mathcal{Z}_h(x,\cdot)$ of $F$ with parameter $q$
exists and is given by the formula
\begin{equation}\label{eq:Z-Feynman-int}
\begin{aligned}
&\int_{C_0[0,T]}^{\mathrm{anf}_q}F(\mathcal Z_h(x,\cdot))\mathfrak{m}(dx)\\
&=\bigg(\prod_{j=1}^n\frac{-iq}{2\pi\|\alpha_jh\|_2^2}\bigg)^{1/2}
\int_{\mathbb R^n}f(u_1,\ldots,u_n)\exp\bigg\{\frac{iq}{2}\sum_{j=1}^n \frac{u_j^2}{\|\alpha_jh\|_2^2}\bigg\}du_1\cdots du_n.
\end{aligned}
\end{equation}
\end{theorem}
\begin{proof}
Using   \eqref{eq:cylinder-x}, \eqref{eq:basic-rel}, and   \eqref{eq:ch-formula}  with $\mathcal G$ replaced  
with $\mathcal Ah$, it follows that for all $\lambda>0$,
\[
\begin{aligned}
&J_{F}(h;\lambda)\\
&=\int_{C_0[0,T]}f (\lambda^{-1/2} \langle{\alpha_1, \mathcal{Z}_h(x,\cdot)}\rangle,
   \ldots,
  \lambda^{-1/2} \langle{\alpha_n, \mathcal{Z}_h(x,\cdot)}\rangle )\mathfrak{m}(dx)\\
&=\int_{C_0[0,T]}f (\lambda^{-1/2} \langle{\alpha_1h,x}\rangle,
   \ldots,
    +\lambda^{-1/2} \langle{\alpha_nh, x}\rangle )\mathfrak{m}(dx)\\
&=\bigg(\prod_{j=1}^n\frac{\lambda}{2\pi\|\alpha_jh\|_2^2}\bigg)^{1/2}
\int_{\mathbb R^n}f (u_1,\ldots,u_n)
\exp\bigg\{-\frac{\lambda}{2}\sum_{j=1}^{n}\frac{u_j^2}{\|\alpha_j h\|_2^2}\bigg\}du_1\cdots du_n.
\end{aligned}
\]
For  $(\lambda,\vec u) \in \mathbb C_+\times \mathbb R^n$, let   
\[
H_{\mathcal A h}(\lambda;\vec u)
=\exp\bigg\{-\frac{\lambda}{2}\sum_{j=1}^{n}\frac{u_j^2}{\|\alpha_j h\|_2^2}\bigg\} 
\]
and for $\lambda\in \mathbb C_+$, let  
\begin{equation}\label{analytic-lambda}
J_{F}^*(h;\lambda)
=\bigg(\prod_{j=1}^n\frac{\lambda}{2\pi\|\alpha_jh\|_2^2}\bigg)^{1/2}
\int_{\mathbb R^n}f (u_1,\ldots,u_n)H_{\mathcal A h}(\lambda;\vec u)du_1\cdots du_n.
\end{equation}
Then we see   that 
\begin{itemize}
  \item[(i)] for all $\lambda>0$, $J_{F}^*(h;\lambda)=J_{F}(h;\lambda)$, 
  \item[(ii)] for each $\lambda \in \mathbb C_+$,  
$H_{\mathcal A h}(\lambda;\vec u)$, as a function of $\vec u$, is an element of $L_p(\mathbb R^n)\cap C_0(\mathbb R^n)$
for all $p\in [1,+\infty]$, and 
  \item[(iii)]  $|H_{\mathcal A h}(\lambda;\vec u)|\le 1$ for all   $(\lambda,\vec u) \in   \widetilde{\mathbb C}_+\times \mathbb R^n$.
\end{itemize}
Thus   using H\"older's inequality, we  can see that 
$|f(\vec u)|H_{\mathcal A h}(\lambda;\vec u)$,
as a function of $\vec u$, is an element of $L_1(\mathbb R^n)$ 
whenever $f\in L_p (\mathbb R^n)$ for every $p\in[1,+\infty]$.
Hence, using the dominated convergence theorem, it follows that
$J_{F}^*(h;\lambda)$
is a continuous function of  $\lambda$ on $\mathbb C_+$.
Clearly, $H_{\mathcal A h}(\lambda;\vec u)$ is analytic on $\mathbb C_+$  
as a function of $\lambda$.
Hence, by the Fubini theorem and the Cauchy theorem, we obtain  that
\[
\int_{\Delta}J_{F}^*(h;\lambda)d\lambda
=\bigg(\prod_{j=1}^n\frac{\lambda}{2\pi\|\alpha_jh\|_2^2} \bigg)^{1/2} 
\int_{\mathbb R^n}f(\vec u )\int_{\Delta}
H_{\mathcal A h}(\lambda;\vec u)d\lambda d\vec u=0
\]
for any rectifiable simple closed curve $\Delta$ lying in $\mathbb C_+$.
Thus by the Morera theorem, 
$\int_{C_0[0,T]}^{\text{\rm anw}_{\lambda}}
F(\mathcal{Z}_h(x,\cdot))\mathfrak{m}(dx)
= J^*(h;\lambda)$
given by \eqref{analytic-lambda}  is an analytic  function of $\lambda$
throughout $\mathbb C_+$. 
Finally,  by the dominated convergence, it follows equation \eqref{eq:Z-Feynman-int}.
\qed\end{proof}

\par
The following observations are very useful to complete the proof of our main theorems 
(i.e., Theorems \ref{thm:by-parts-F}, \ref{thmby-parts-gfft1}, and \ref{thmby-parts-gfft2}  below).
\begin{itemize}
  \item[(1)]
If we choose $z\in L_2[0,T]$ and define 
\begin{equation}\label{eq:function-wz}
w_z(t)=\int_0^t z(s)ds 
\end{equation} 
for $t\in[0,T]$, then $w_z$ is an element of $C_0'[0,T]$, see equation \eqref{eq:CMspace}, $Dw_z=z$ $m_L$-a.e. on $[0,T]$,
where $Dw(s)=\frac{dw}{ds}(s)$,
and for all $v\in L_2[0,T]$, 
\begin{equation}\label{eq:basic-rel2} 
\langle{v,w_z}\rangle=(v,Dw_z)_2=(v,z)_2,
\end{equation}
where of course $(v,z)_2=\int_0^T v(s)z(s)ds$.
 \item[(2)]
Let $h_1$ and $h_2$ be functions in $\mathrm{Supp}_{\infty}[0,T]$.
Given $z\in L_2[0,T]$, 
let $w_{zh_1}\in C_0'[0,T]$ be given by \eqref{eq:function-w-phih} above.
In this case, using \eqref{eq:basic-rel} and \eqref{eq:basic-rel2},
we then also  see that 
\[
\langle{v,\mathcal Z_{h_2}(w_{zh_1},\cdot)}\rangle
=\langle{vh_2, w_{zh_1}}\rangle
=(vh_2,zh_1)_2. 
\] 
 \item[(3)]  Given an orthogonal set $\mathcal A=\{\alpha_1,\ldots,\alpha_n\}$   
of nonzero functions in $L_2[0,T]$,  $p\in[1,+\infty]$, and a nonnegative integer $m$,
let $F\in \mathcal B_{\mathcal A}(p;m)$.
Then, using \eqref{eq:cylinder-x} and  \eqref{eq:basic-rel},
we see that for any function $h\in \mathcal O_{\mathrm{Supp}_{\infty}}(\mathcal A)$,
 $F(\mathcal Z_h(x,\cdot))$ belongs to the space $\mathcal B_{\mathcal Ah}(p;m)$.
\end{itemize}

Our next lemma follows directly from the definitions of $\delta_{h_1,h_2} F $ and $\mathcal B_{\mathcal A}(p;m)$.


\begin{lemma}\label{lemma:1}
Let $\mathcal A=\{\alpha_1,\ldots,\alpha_n\}$ be an orthogonal set of nonzero 
functions in $L_2[0,T]$, let $p\in[1,+\infty]$ be given, let $m$ be a positive integer, 
let $F\in \mathcal B_{\mathcal A}(p;m)$ be given by equation \eqref{eq:cylinder-x},  
and let $w_z\in C_0'[0,T]$ be  given by \eqref{eq:function-wz}.
Then for any functions $h_1\in \mathcal{O}_{\mathrm{Supp}_{\infty}} (\mathcal{A})$
and $h_2\in  \mathrm{Supp}_{\infty}[0,T]$,
\begin{equation}\label{eq:ev-delta}
\begin{aligned}
 \delta_{h_1,h_2}F(x|w_z) 
&\equiv \delta F(\mathcal Z_{h_1}(x,\cdot)|\mathcal Z_{h_2}(w_z,\cdot))\\
&=\sum_{j=1}^n (\alpha_j h_2,z)_2 f_j(\langle{\alpha_1h_1,x}\rangle, 
\ldots ,\langle{\alpha_n h_1,x}\rangle) 
\end{aligned}
\end{equation}
for s-a.e. $x\in C_0[0,T]$.
Furthermore, $\delta_{h_1,h_2}F(\cdot|w_z) \in \mathcal B_{\mathcal A h_1}(p;m-1)$.
\end{lemma}


\begin{lemma}\label{lemma:B}
Let $\mathcal A$, $p$, $m$, $F$, and $w_z$ be as in Lemma \ref{lemma:1}. 
For $p' \in [1,+\infty]$ with $(1/p)+(1/p')=1$, let $G\in \mathcal B_{\mathcal A}(p';m)$
be given by 
\begin{equation}\label{eq:cylinder-G}
G(x)=g (\langle{\alpha_1,x}\rangle,\ldots,\langle{\alpha_n,x}\rangle).
\end{equation}
Define $R(x)=F(x)G(x)$ for $x\in C_0[0,T]$.
Then  
$R \in \mathcal B_{\mathcal A}(1;m)$, and for any functions 
$h_1\in \mathcal{O}_{\mathrm{Supp}_{\infty}} (\mathcal{A})$
and $h_2\in  \mathrm{Supp}_{\infty}[0,T]$,
$\delta_{h_1,h_2} R(\cdot|w_z) \in \mathcal B_{\mathcal Ah_1}(1;m-1)$, as a function of $x$.
\end{lemma}
\begin{proof}
Note that $R(x)=r(\langle{\alpha_1,x}\rangle,\ldots,\langle{\alpha_n,x}\rangle)$
where 
\[
r(u_1,\ldots,u_n)=f(u_1,\ldots,u_n)g (u_1,\ldots,u_n).
\]
We see that $R$ is an element of $\mathcal B_{\mathcal A}(1;m)$
since all the $k$-th order partial derivatives of $r$ are continuous 
and in $L_1(\mathbb R^n)$ for $k\in\{0,1,\ldots,m\}$ by H\"older's inequality.
The fact that $\delta R_{h_1,h_2}(x|w_z)$, as a function of $x$, is an element 
of $\mathcal B_{\mathcal Ah_1}(1;m-1)$
now follows from Lemma \ref{lemma:1}.
\qed\end{proof}

\begin{remark}
Given $z\in L_2[0,T]$ and $h_1\in  \mathcal{O}_{\mathrm{Supp}_{\infty}} (\mathcal{A})$,
let $w_{zh_1}\in C_0'[0,T]$ be given by \eqref{eq:function-w-phih}.
Then  equation \eqref{eq:ev-delta} with $w_z$ replaced with $w_{zh_1}$ can be rewritten as
\begin{equation}\label{eq:ev-delta-h12}
\begin{aligned}
\delta_{h_1,h_2}F(x|w_{zh_1}) 
&\equiv \delta F(\mathcal Z_{h_1}(x,\cdot)|\mathcal Z_{h_2}(w_{zh_1},\cdot))\\&
=\sum_{j=1}^n (\alpha_j h_2,zh_1)_2 
f_j(\langle{\alpha_1h_1,x}\rangle, \ldots ,\langle{\alpha_n h_1,x}\rangle) 
\end{aligned}
\end{equation}
for   s-a.e. $x\in C_0[0,T]$.
\end{remark}

\par
In our next theorem we obtain an integration by parts formula for the generalized analytic Feynman integral.


\begin{theorem}\label{thm:by-parts-F}
Let $\mathcal A$,  $p$, $m$, $F$, and $G$  be as in Lemma \ref{lemma:B}.
Given $z\in L_2[0,T]$ and $h_1\in  \mathcal{O}_{\mathrm{Supp}_{\infty}} (\mathcal{A})$,
let $w_{zh_1}\in C_0'[0,T]$  be given by \eqref{eq:function-w-phih}.
Then for any function $h_2$ in   $\mathrm{Supp}_{\infty}[0,T]$, 
and all $q\in \mathbb R \setminus\{0\}$, it follows that 
\begin{equation}\label{eq:by-parts-F}
\begin{aligned}
&\int_{C_0[0,T]}^{\mathrm{anf}_q}
\big[F(\mathcal Z_{h_1}(x,\cdot))\delta G(\mathcal Z_{h_1}(x,\cdot)|\mathcal Z_{h_2}(w_{zh_1},\cdot)) \\
& \qquad\qquad
+\delta F(\mathcal Z_{h_1}(x,\cdot)|\mathcal Z_{h_2}(w_{zh_1},\cdot))
G(\mathcal Z_{h_1}(x,\cdot))\big]\mathfrak{m}(dx)\\
&=-iq \int_{C_0[0,T]}^{\mathrm{anf}_q} \langle{z,\mathcal Z_{h_2}(x,\cdot)}\rangle
F(\mathcal Z_{h_1}(x,\cdot))G(\mathcal Z_{h_1}(x,\cdot)) \mathfrak{m}(dx).
\end{aligned}
\end{equation}
\end{theorem}
\begin{proof}
Again define $R(x)\equiv F(x)G(x)$ for $x\in C_0[0,T]$ and let 
$r(u_1,\ldots,u_n)\equiv f(u_1,\ldots,u_n)g(u_1,\ldots,u_n)$.
Then as noted in Lemma \ref{lemma:B} and its proof,
$R \in \mathcal B_{\mathcal A}(1;m)$, $\delta_{h_1,h_2} R (\cdot|w_{zh_1}) \in \mathcal B_{\mathcal Ah_1}(1;m-1)$, 
and  all the $k$-th order partial derivatives of $r$ are continuous and in $L_1(\mathbb R^n)$ 
for $k\in\{0,1,\ldots, m\}$. Hence $R(\rho x)$  is Wiener integrable on $C_0[0,T]$ for each $\rho>0$.
In addition, applying \eqref{eq:ev-delta-h12}, it follows that  for s-a.e. $x\in C_0[0,T]$,
\begin{equation}\label{eq:ev001}
\begin{aligned}
&\delta_{h_1,h_2} R(x|w_{zh_1}) 
   \equiv   \delta R(\mathcal Z_{h_1}(x,\cdot)|\mathcal Z_{h_2}(w_{zh_1},\cdot))\\
&=  F(\mathcal Z_{h_1}(x,\cdot))\delta G(\mathcal Z_{h_1}(x,\cdot)|\mathcal Z_{h_2}(w_{zh_1},\cdot))\\
&\qquad\qquad
+ \delta F(\mathcal Z_{h_1}(x,\cdot)|\mathcal Z_{h_2}(w_{zh_1},\cdot)) G(\mathcal Z_{h_1}(x,\cdot))\\
&=  f(\langle{\alpha_1 h_1,x}\rangle,\ldots, \langle{\alpha_n h_1,x}\rangle)
\sum_{j=1}^n  (\alpha_j h_2,zh_1)_2  g_j(\langle{\alpha_1 h_1,x}\rangle,\ldots, \langle{\alpha_n h_1,x}\rangle)\\
&\quad 
 +g(\langle{\alpha_1 h_1,x}\rangle,\ldots, \langle{\alpha_n h_1,x}\rangle)
\sum_{j=1}^n (\alpha_j h_2,zh_1)_2 f_j(\langle{\alpha_1 h_1,x}\rangle,\ldots, \langle{\alpha_n h_1,x}\rangle). 
\end{aligned}
\end{equation}
Now since $fg_j$ and $gf_l$ are all continuous and in $L_1(\mathbb R^n)$  for $j,l\in\{1,2,\ldots,n\}$,
it  is quite easy to see that $\delta R(\rho \mathcal Z_{h_1}(x,\cdot)|\rho \mathcal Z_{h_2}(w_{zh_1},\cdot))$,
as a function of $x$, is Wiener integrable for all $\rho>0$.
In addition, $\delta R(\mathcal Z_{h_1}(x,\cdot)|\mathcal Z_{h_2}(w_{zh_1},\cdot))$ is analytic Feynman integrable which 
can be seen by integrating the right-hand side of \eqref{eq:ev001} term by term. 
For example, applying \eqref{eq:ch-formula}, it follows that for any $j\in\{1,\ldots,n\}$,
\[
\begin{aligned}
&\int_{C_0[0,T]}^{\mathrm{anf}_q}
f(\langle{\alpha_1 h_1,x}\rangle,\ldots, \langle{\alpha_n h_1,x}\rangle)
(\alpha_j h_2, zh_1)_2\\
&\qquad\quad \times
  g_j(\langle{\alpha_1 h_1,x}\rangle,\ldots, \langle{\alpha_n h_1,x}\rangle)
\mathfrak{m}(dx)\\
&=(\alpha_j h_2, zh_1)_2 \bigg(\prod_{j=1}^n\frac{-iq}{2\pi\|\alpha_j h_1\|_2^2}\bigg)\\
&\qquad \times
\int_{\mathbb R^n}f(u_1,\ldots,u_n)g_j(u_1,\ldots,u_n)\exp\bigg\{\frac{iq}{2}\sum_{j=1}^n \frac{u_j^2}{\|\alpha_j h_1\|_2^2}\bigg\}du_1\cdots d u_n
\end{aligned}
\]
since $f(u_1,\ldots,u_n)g_j(u_1,\ldots,u_n)$ is continuous and in $L_1(\mathbb R^n)$.
Thus \eqref{eq:by-parts-F} follows from Theorem \ref{byparts-step1} above.
\qed\end{proof}

By choosing $p=2$ and $F=G$ in Theorem \ref{thm:by-parts-F},
we obtain the following corollary.


\begin{corollary}
Let $\mathcal A$, $m$, $z$, $h_1$,  and $w_{zh_1}$ be as in Theorem \ref{thm:by-parts-F}.
Let $F\in \mathcal B_{\mathcal A}(2;m)$ be given by \eqref{eq:cylinder-x}.
Then for any function $h_2$ in   $\mathrm{Supp}_{\infty}[0,T]$, 
and all $q\in \mathbb R \setminus\{0\}$, it follows that 
\[
\begin{aligned}
&\int_{C_0[0,T]}^{\mathrm{anf}_q}
F(\mathcal Z_{h_1}(x,\cdot))
\delta F(\mathcal Z_{h_1}(x,\cdot)|\mathcal Z_{h_2}(w_{zh_1},\cdot)) \mathfrak{m}(dx)\\
&=-\frac{iq}{2} \int_{C_0[0,T]}^{\mathrm{anf}_q} \langle{z,\mathcal Z_{h_2}(x,\cdot)}\rangle
\big[F(\mathcal Z_{h_1}(x,\cdot))\big]^2 \mathfrak{m}(dx).
\end{aligned}
\]
\end{corollary}

\setcounter{equation}{0}
\section{Integration by parts formulas  involving generalized analytic Fourier--Feynman transforms}\label{sec:parts-fft}

In this section we establish integration by parts formulas involving
the $\mathcal Z_h$-GFFTs. We start this section with the existence theorem 
of the $L_p$ analytic $\mathcal Z_h$-GFFT of functionals  $F$ in 
$\mathcal B_{\mathcal A}(p;m)$.


\begin{theorem}\label{thm:tpq-class}
Let $\mathcal A=\{\alpha_1,\ldots,\alpha_n\}$ be an orthogonal set of nonzero functions in $L_2[0,T]$,
let $p\in[1,2]$ be given, let $m$ be a non-negative  integer, and 
let $F\in \mathcal B_{\mathcal A}(p;m)$ be given by equation \eqref{eq:cylinder-x}.
Then for any $k\in  \mathcal{O}_{\mathrm{Supp}_{\infty}}(\mathcal{A})$, 
and all $q\in \mathbb R \setminus\{0\}$,
the $L_p$  analytic $\mathcal Z_k$-GFFT $T_{q,k}^{(p)}(F)$ of $F$ exists 
and   is given by the formula
\begin{equation}\label{eq:gfft}
\begin{aligned}
T_{q,k}^{(p)}(F)(y)
&= \bigg(\prod_{j=1}^n\frac{-iq}{2\pi\|\alpha_jk\|_2^2}\bigg)^{1/2}
\int_{\mathbb R^n}f(u_1,\ldots,u_n)\\
& \qquad\qquad\qquad\times
\exp\bigg\{\frac{iq}{2}\sum_{j=1}^n \frac{[u_j-\langle{\alpha_j,y}\rangle]^2}{\|\alpha_jk\|_2^2}\bigg\}du_1\cdots du_n\\
&= \bigg(\prod_{j=1}^n\frac{-iq}{2\pi\|\alpha_jk\|_2^2}\bigg)^{1/2}
\int_{\mathbb R^n}f(u_1+\langle{\alpha_1,y}\rangle,\ldots,u_n+\langle{\alpha_n,y}\rangle)\\
& \qquad\qquad\qquad\times
\exp\bigg\{\frac{iq}{2}\sum_{j=1}^n \frac{u_j^2}{\|\alpha_jk\|_2^2}\bigg\}du_1\cdots du_n\\
\end{aligned}
\end{equation}
for s-a.e. $y\in C_0[0,T]$.
Furthermore, $T_{q,k}^{(p)}(F)\in  \mathcal  B_{\mathcal A}(p';m)$ where $(1/p)+(1/p')=1$.
\end{theorem}
\begin{proof}
First, in the case $m=0$,
the proof given in \cite[Theorems 4.7 and 4.8]{CC18-BJMA}  with the current 
hypotheses on $F$ and $k$ also works here.
Now let an orthogonal set $\mathcal A=\{\alpha_1,\ldots,\alpha_n\}$  
of   nonzero functions in $L_2[0,T]$, $p\in[1,2]$, $m \in \{1,2,\ldots\}$   be given and 
let $F\in \mathcal  B_{\mathcal A}(p;m)$.
Since $\mathcal  B_{\mathcal A}(p;m)\subset \mathcal  B_{\mathcal A}(p;0)$,
we know that    $T_{q,k}^{(p)}(F)$ exists and is given by equation
\eqref{eq:gfft}.
The proof that $T_{q,k}^{(p)}(F)$ belongs to $\mathcal  B_{\mathcal A}(p';m)$ 
for $m\ge 1$ is similar to the proof in \cite{CC18-BJMA} for the case $m=0$.
\qed\end{proof}

In view of Theorems \ref{thm:by-parts-F} and \ref{thm:tpq-class}, we get the following corollary.


\begin{corollary}\label{coro:condition-add}
Let $\mathcal A=\{\alpha_1,\ldots,\alpha_n\}$ be an orthogonal set of nonzero functions in $L_2[0,T]$,
let $m$ be a positive  integer, and given $z\in L_2[0,T]$ and $h_1\in  \mathcal{O}_{\mathrm{Supp}_{\infty}} (\mathcal{A})$,
let $w_{zh_1}\in C_0'[0,T]$  be given by \eqref{eq:function-w-phih}.
Let $F$ and $G$  in $\mathcal B_{\mathcal A}(2;m)$ be given by \eqref{eq:cylinder-x} and \eqref{eq:cylinder-G},
respectively.  
Then  any functions  $k_1, k_2 \in \mathcal{O}_{\mathrm{Supp}_{\infty}}(\mathcal{A})$, 
$h_2 \in \mathrm{Supp}_{\infty}[0,T]$, and any $q_1, q_2, q_3 \in \mathbb R \setminus\{0\}$,
it follows that 
\[
\begin{aligned}
&\int_{C_0[0,T]}^{\mathrm{anf}_{q_3}}
\big[T_{q_1,k_1}^{(2)}(F)(\mathcal Z_{h_1}(x,\cdot))\delta T_{q_2,k_2}^{(2)}(G)(\mathcal Z_{h_1}(x,\cdot))|\mathcal Z_{h_2}(w_{zh_1},\cdot)) \\
& \qquad\qquad
+\delta T_{q_1,k_1}^{(2)}(F)(\mathcal Z_{h_1}(x,\cdot)|\mathcal Z_{h_2}(w_{zh_1},\cdot))
T_{q_2,k_2}^{(2)}(G)(\mathcal Z_{h_1}(x,\cdot))\big]\mathfrak{m}(dx)\\
&=-iq_3 \int_{C_0[0,T]}^{\mathrm{anf}_{q_3}} \langle{z,\mathcal Z_{h_2}(x,\cdot)}\rangle
T_{q_1,k_1}^{(2)}(F)(\mathcal Z_{h_1}(x,\cdot))T_{q_2,k_1}^{(2)}(G)(\mathcal Z_{h_1}(x,\cdot)) \mathfrak{m}(dx).
\end{aligned}
\]
\end{corollary}

In our next theorem we show that the transform with respect to the first argument 
of the variation equals the variation of the transform.


\begin{theorem}\label{thm:tpq}
Let $\mathcal A=\{\alpha_1,\ldots,\alpha_n\}$ be an orthogonal set of nonzero functions in $L_2[0,T]$,
let $p\in[1,2]$ be given, let $m$ be a positive  integer,   
and  let $F\in \mathcal B_{\mathcal A}(p;m)$ be given by equation \eqref{eq:cylinder-x}.
Also, let $w_z\in C_0'[0,T]$ be given by \eqref{eq:function-wz} above. 
Then  for any functions $h_1 \in   \mathcal{O}_{\mathrm{Supp}_{\infty}}(\mathcal{A})$,
$\{h_2, k\}\subset\mathrm{Supp}_{\infty}[0,T]$ with $kh_1 \in  \mathcal{O}_{\mathrm{Supp}_{\infty}}(\mathcal{A})$
(or $k  \in  \mathcal{O}_{\mathrm{Supp}_{\infty}}(\mathcal{A}h_1 )$),  
 all $q\in \mathbb R \setminus\{0\}$, and   s-a.e. $y\in C_0[0,T]$, it follows that 
\begin{equation}\label{eq:tpq-var=vat-tpq}
T_{q,k}^{(p)}(\delta_{h_1,h_2}F (\cdot|w_z)) 
=\delta_{h_1,h_2} T_{q,kh_1}^{(p)}(F)(y|w_z)
\end{equation}
which, as a function of $y$, is an element of $\mathcal B_{\mathcal Ah_1}(p';m-1)$.
Also, both expressions 
in \eqref{eq:tpq-var=vat-tpq} are given by the expression
\begin{equation}\label{eq:gfft-delta}
\begin{aligned}
&\bigg(\prod_{j=1}^n\frac{-iq}{2\pi\|\alpha_jkh_1\|_2^2}\bigg)^{1/2}
\int_{\mathbb R^n}\bigg[\sum_{j=1}^n (\alpha_j h_2,z)_2 \\
& \,\, \times
f_j(u_1+\langle{\alpha_1h_1,y}\rangle,\ldots,u_n+\langle{\alpha_nh_1,y}\rangle)\bigg] 
\exp\bigg\{\frac{iq}{2}\sum_{j=1}^n \frac{u_j^2}{\|\alpha_jkh_1\|_2^2}\bigg\}du_1\cdots du_n.
\end{aligned}
\end{equation} 
 \end{theorem}
\begin{proof}
First, using \eqref{eq:1st} with $F$, $x$ and $w$  replaced with $T_{q,kh_1}^{(p)}(F)$,
$y$ and $w_z$, and \eqref{eq:gfft} with $k$ and $y$ replaced with $kh_1$ and 
$\mathcal Z_{h_1}(y,\cdot)$, respectively, \eqref{eq:basic-rel}, 
and \eqref{eq:basic-rel2},
we obtain that
\begin{equation}\label{eq:folland}
\begin{aligned}
 &\delta_{h_1,h_2} T_{q,kh_1}^{(p)}(F)(y|w_z) 
 \equiv \delta T_{q,kh_1}^{(p)}(F)(\mathcal Z_{h_1}(y,\cdot))|\mathcal Z_{h_2}(w_z,\cdot))\\
&=\frac{\partial}{\partial \mu}
 T_{q,kh_1}^{(p)}(F)(\mathcal Z_{h_1}(y,\cdot)+\mu \mathcal Z_{h_2}(w_z,\cdot)) \bigg|_{\mu=0}\\
&= \bigg(\prod_{j=1}^n\frac{-iq}{2\pi\|\alpha_jkh_1\|_2^2}\bigg)^{1/2}
\int_{\mathbb R^n}
\frac{\partial}{\partial \mu}f(u_1+\langle{\alpha_1h_1,y}\rangle+\mu\langle{\alpha_1h_2,w_z}\rangle,\\
&\qquad\qquad\qquad\qquad\qquad\qquad
\ldots,u_n+\langle{\alpha_nh_1,y}\rangle+\mu\langle{\alpha_nh_2,w_z}\rangle)\bigg|_{\mu=0}\\
& \qquad\qquad\qquad\qquad\qquad\times
\exp\bigg\{\frac{iq}{2}\sum_{j=1}^n \frac{u_j^2}{\|\alpha_jkh_1\|_2^2}\bigg\} du_1\cdots du_n\\
&= \bigg(\prod_{j=1}^n\frac{-iq}{2\pi\|\alpha_jkh_1\|_2^2}\bigg)^{1/2}
\int_{\mathbb R^n}\bigg[\sum_{j=1}^n (\alpha_j h_2,z)_2 \\
& \qquad\qquad\qquad\qquad\times
f_j(u_1+\langle{\alpha_1h_1,y}\rangle,\ldots,u_n+\langle{\alpha_nh_1,y}\rangle)\bigg] \\
& \qquad\qquad\qquad\qquad\times
\exp\bigg\{\frac{iq}{2}\sum_{j=1}^n \frac{u_j^2}{\|\alpha_jkh_1\|_2^2}\bigg\}du_1\cdots du_n.\\
\end{aligned}
\end{equation}
The second equality of \eqref{eq:folland} follows from the fact 
that $T_{q,k}^{(p)}(F)$ is in  $\mathcal  B_{\mathcal A}(p';m)$,
and Theorem 2.27 in \cite{folland}.

Next, using \eqref{eq:ev-delta}, it follows that
\[
\begin{aligned}
& T_{q,k}^{(p)}(\delta_{h_1,h_2} F(\cdot|w_z))(y)\\
&=\int_{C_0[0,T]}^{\mathrm{anf}_q}\delta_{h_1,h_2} F(y+\mathcal Z_k(x,\cdot)|w_z) \mathfrak m( dx)\\
&=\int_{C_0[0,T]}^{\mathrm{anf}_q}\delta  F(\mathcal Z_{h_1}( y+\mathcal Z_k(x,\cdot),\cdot)
|\mathcal Z_{h_2}(w_z,\cdot))\mathfrak m( dx)\\
&=\int_{C_0[0,T]}^{\mathrm{anf}_q} 
\sum_{j=1}^n(\alpha_j h_2,z)_2\\
&\qquad \times 
f_j(\langle{\alpha_1h_1,y+\mathcal Z_k(x,\cdot) }\rangle, 
\ldots ,\langle{\alpha_n h_1,y+\mathcal Z_k(x,\cdot) }\rangle) 
\mathfrak m( dx)\\
&=\int_{C_0[0,T]}^{\mathrm{anf}_q} 
\sum_{j=1}^n(\alpha_j h_2,z)_2 \\
&\qquad \times 
f_j(\langle{\alpha_1h_1,y}\rangle +\langle{\alpha_1kh_1,x}\rangle, 
\ldots ,\langle{\alpha_nh_1,y}\rangle +\langle{\alpha_nkh_1,x}\rangle)  
\mathfrak m( dx).
\end{aligned}
\]
 Then, evaluating the above analytic Feynman integral
together with use of \eqref{eq:ch-formula} with $\mathcal G$ replaced 
with $\mathcal Akh_1$, we obtain equation \eqref{eq:tpq-var=vat-tpq} for
s-a.e. $y\in C_0[0,T]$. Finally, $\delta_{h_1,h_2} T_{q,kh_1}^{(p)}(F)(\cdot|w)$
is an element of  $\mathcal B_{\mathcal Ah_1}(p';m-1)$,
since $T_{q,kh_1}^{(p)}(F)(\cdot|w)$ is an element of $\mathcal B_{\mathcal A}(p';m)$.
\qed\end{proof}


\begin{remark}
A close examination of the proof of   Theorem \ref{thm:tpq}
shows that $\delta_{h_1,h_2} T_{q,k}^{(p)}(F)(y|w_z)$
is an element of $\mathcal B_{\mathcal Ah_1}(p';m-1)$ and is given by 
the expression  \eqref{eq:gfft-delta} with $kh_1$ replaced with
$k$ for s-a.e. $y\in C_0'[0,T]$.
\end{remark}

\par
In our next theorems we obtain   parts formulas 
involving GFFTs.


\begin{theorem}\label{thmby-parts-gfft1}
Let $\mathcal A$, $m$,  $z$, $h_1$,   $w_{zh_1}$,  $F$, and $G$ 
be as in Corollary \ref{coro:condition-add}.
Then  for any functions  
$h_2, k\in\mathrm{Supp}_{\infty}[0,T]$ with $kh_1 \in  \mathcal{O}_{\mathrm{Supp}_{\infty}}(\mathcal{A})$
(or $k  \in  \mathcal{O}_{\mathrm{Supp}_{\infty}}(\mathcal{A}h_1 )$), 
 and  all  $q_1, q_2, q_3 \in \mathbb R \setminus\{0\}$, 
it follows that 
\begin{equation} \label{eq:by-parts-gfft}
\begin{aligned}
&\int_{C_0[0,T]}^{\mathrm{anf}_{q_3}}
\big[ T_{q_1,kh_1}^{(2)}(F)(\mathcal Z_{h_1}(x,\cdot))
\delta T_{q_2,kh_1}^{(2)}(G)(\mathcal Z_{h_1}(x,\cdot)|\mathcal Z_{h_2}(w_{zh_1},\cdot)) \\
& \qquad\qquad
+\delta T_{q_1,kh_1}^{(2)}(F)(\mathcal Z_{h_1}(x,\cdot)|\mathcal Z_{h_2}(w_{zh_1},\cdot))
T_{q_2,kh_1}^{(2)}(G) (\mathcal Z_{h_1}(x,\cdot))\big]\mathfrak{m}(dx)\\
&=-iq_3 \int_{C_0[0,T]}^{\mathrm{anf}_{q_3}} \langle{z,\mathcal Z_{h_2}(x,\cdot)}\rangle
T_{q_1,kh_1}^{(2)}(F)(\mathcal Z_{h_1}(x,\cdot))T_{q_2,kh_1}^{(2)}(G)(\mathcal Z_{h_1}(x,\cdot)) \mathfrak{m}(dx).
\end{aligned}
\end{equation}
\end{theorem}
\begin{proof}
For $x\in C_0[0,T]$, let $R(x)=T_{q_1,kh_1}^{(2)}(F)(x)T_{q_2,kh_1}^{(2)}(G)(x)$.
Then by Theorem \ref{thm:tpq-class},
$T_{q_1,kh_1}^{(2)}(F)$ and $T_{q_2,kh_1}^{(2)}(G)$ are in $\mathcal B_{\mathcal A}(2;m)$. 
Therefore, by Lemma \ref{lemma:B}, $R$ is in $\mathcal B_{\mathcal A}(1;m)$. 
Moreover, by Lemma \ref{lemma:1}, $\delta_{h_1,h_2}R(x|w_{zh_1})$,
as a function of $x$, is an element of $\mathcal B_{\mathcal Ah_1}(1;m-1)$. 
Thus equation \eqref{eq:by-parts-gfft} follows from Theorem \ref{thm:by-parts-F}
with $F$ and $G$ replaced with $T_{q_1,kh_1}^{(2)}(F)$ and $T_{q_2,kh_1}^{(2)}(G)$, respectively.
\qed\end{proof}


\begin{theorem}\label{thmby-parts-gfft2} 
Let $\mathcal A$,  $p$, $m$,   $F$, and $w_z$  be as in Theorem \ref{thm:tpq}.
Also, let  $G\in  \mathcal B_{\mathcal A}(p;m)$ be given by \eqref{eq:cylinder-G}.
Then  for any functions  
$h_2, k\in\mathrm{Supp}_{\infty}[0,T]$ with $kh_1 \in  \mathcal{O}_{\mathrm{Supp}_{\infty}}(\mathcal{A})$
(or $k  \in  \mathcal{O}_{\mathrm{Supp}_{\infty}}(\mathcal{A}h_1 )$), 
 and  all $q_1,q_2 \in \mathbb R \setminus\{0\}$,
it follows that 
\begin{equation} \label{eq:by-parts-gfft2}
\begin{aligned}
&
\int_{C_0[0,T]}^{\mathrm{anf}_{q_2} }
\big[ F (\mathcal Z_{h_1}(x,\cdot))
\delta T_{q_1,kh_1}^{(p)}(G)(\mathcal Z_{h_1}(x,\cdot)|\mathcal Z_{h_2}(w_{zh_1},\cdot)) \\
& \qquad\qquad
+\delta F (\mathcal Z_{h_1}(x,\cdot)|\mathcal Z_{h_2}(w_{zh_1},\cdot))
T_{q_1,kh_1}^{(p)}(G)(\mathcal Z_{h_1}(x,\cdot))\big]\mathfrak{m}(dx)\\
&=-iq_2 \int_{C_0[0,T]}^{\mathrm{anf}_{q_2}} \langle{z,\mathcal Z_{h_2}(x,\cdot)}\rangle
 F (\mathcal Z_{h_1}(x,\cdot))T_{q_1,kh_1}^{(p)}(G)(\mathcal Z_{h_1}(x,\cdot)) \mathfrak{m}(dx).
\end{aligned}
\end{equation}
\end{theorem}
\begin{proof}
 For $x\in C_0[0,T]$, let $R(x)= F (x)T_{q_1,kh_1}^{(p)}(G)(x)$.
Again, since $T_{q_1,kh_1}^{(p)}(G)$ is in $\mathcal B_{\mathcal A}(p';m)$ and
$F$ is in $\mathcal B_{\mathcal A}(p;m)$, it follows that $R$ belongs to $\mathcal B_{\mathcal A}(1;m)$,
and as a function of $x$, $\delta_{h_1,h_2}R(x|w_{zh_1})$ belongs to $\mathcal B_{\mathcal A}(1;m-1)$. 
Thus equation \eqref{eq:by-parts-gfft2} follows from Theorem \ref{thm:by-parts-F}
with $G$ replaced with $T_{q_1,kh_1}^{(p)}(G)$.
\qed\end{proof}

\par
Choosing $F=G$ in Theorem  \ref{thmby-parts-gfft1} 
above, we obtain the following corollary.


\begin{corollary}
Let $\mathcal A$, $m$, $z$, $h_1$, $w_{zh_1}$, and $F$ be as in Corollary  \ref{coro:condition-add}.
Then  for any functions  
$h_2, k\in\mathrm{Supp}_{\infty}[0,T]$ with $kh_1 \in  \mathcal{O}_{\mathrm{Supp}_{\infty}}(\mathcal{A})$
(or $k  \in  \mathcal{O}_{\mathrm{Supp}_{\infty}}(\mathcal{A}h_1 )$), 
 and  all $q_1,q_2 \in \mathbb R \setminus\{0\}$,
it follows that 
\[
\begin{aligned}
&\int_{C_0[0,T]}^{\mathrm{anf}_{q_2}}
 T_{q_1,kh_1}^{(2)}(F)(\mathcal Z_{h_1}(x,\cdot))
\delta T_{q_1,kh_1}^{(2)}(F)(\mathcal Z_{h_1}(x,\cdot)|\mathcal Z_{h_2}(w_{zh_1},\cdot))  
 \mathfrak{m}(dx)\\
&=-\frac{iq_2}{2} \int_{C_0[0,T]}^{\mathrm{anf}_{q_2}} \langle{z,\mathcal Z_{h_2}(x,\cdot)}\rangle
\big[T_{q_1,kh_1}^{(2)}(F)(\mathcal Z_{h_1}(x,\cdot))\big]^2 \mathfrak{m}(dx).
\end{aligned}
\]
\end{corollary}

\par
Choosing $F=G$ in Theorem  \ref{thmby-parts-gfft2} 
above, we  also obtain the following corollary.


\begin{corollary}
Let $\mathcal A$, $p$, $m$, $F$, and $w_{zh_1}$ be as in Theorem \ref{thm:tpq}.
Then  for any functions  
$h_2, k\in\mathrm{Supp}_{\infty}[0,T]$ with $kh_1 \in  \mathcal{O}_{\mathrm{Supp}_{\infty}}(\mathcal{A})$
(or $k  \in  \mathcal{O}_{\mathrm{Supp}_{\infty}}(\mathcal{A}h_1 )$), 
 and  all $q_1,q_2\in \mathbb R \setminus\{0\}$,
it follows that 
\[
\begin{aligned}
&\int_{C_0[0,T]}^{\mathrm{anf}_{q_2}}
\big[ F (\mathcal Z_{h_1}(x,\cdot))
\delta T_{q_1,kh_1}^{(p)}(F )(\mathcal Z_{h_1}(x,\cdot)|\mathcal Z_{h_2}(w_{zh_1},\cdot)) \\
& \qquad\qquad
+\delta F (\mathcal Z_{h_1}(x,\cdot)|\mathcal Z_{h_2}(w_{zh_1},\cdot))
T_{q_1,kh_1}^{(p)}(F)(\mathcal Z_{h_1}(x,\cdot))\big]\mathfrak{m}(dx)\\
&=-iq_2 \int_{C_0[0,T]}^{\mathrm{anf}_{q_2}} \langle{z,\mathcal Z_{h_2}(x,\cdot)}\rangle
 F (\mathcal Z_{h_1}(x,\cdot))T_{q_1,kh_1}^{(p)}(F)(\mathcal Z_{h_1}(x,\cdot)) \mathfrak{m}(dx).
\end{aligned}
\]
\end{corollary}


\end{document}